\documentclass[12pt]{amsart}
\usepackage{amsthm,amssymb,verbatim,amsmath, amsfonts, amscd}

\usepackage{amsthm,amssymb,verbatim}
\usepackage{array}
\usepackage{hyperref}
\usepackage{verbatim}

\usepackage{comment}
\theoremstyle{plain}\newtheorem{Theorem}{Theorem}[section]
\theoremstyle{plain}\newtheorem{Corollary}[Theorem]{Corollary}
\theoremstyle{plain}\newtheorem{Lemma}[Theorem]{Lemma}
\theoremstyle{definition}\newtheorem{Notation}[Theorem]{Notation}

\begin{document}
\title{The Brauer indecomposability of Scott modules with semidihedral vertex}
\date{\today}
\author{{SHIGEO KOSHITANI AND {\.I}PEK TUVAY}}
\address{Center for Frontier Science,
Chiba University, 1-33 Yayoi-cho, Inage-ku, Chiba 263-8522, Japan.}
\email{koshitan@math.s.chiba-u.ac.jp}
\address{Mimar Sinan Fine Arts University, Department of Mathematics, 34380, Bomonti, \c{S}i\c{s}li, Istanbul, Turkey}
\email{ipek.tuvay@msgsu.edu.tr}

\thanks
{Keywords: Brauer indecomposability; Scott modules;  Brauer construction; Semidihedral group
\\ \indent
{2010 {\it Mathematics subject classification:} 20C20, 20C05, 20C15}}
\maketitle

\begin{abstract}  We present a sufficient condition for the
$kG$-Scott module with vertex $P$ to remain
indecomposable under the Brauer construction for any subgroup $Q$ of $P$
as $k[Q\,C_G(Q)]$-module,
where $k$ is a field of characteristic $2$,
and $P$ is a semidihedral $2$-subgroup of a finite group $G$.
This generalizes results for the cases where $P$ is abelian
or dihedral.
The Brauer indecomposability  is defined
\linebreak
by R.~Kessar, N.~Kunugi and N.~Mitsuhashi.
The motivation of
\linebreak
this paper is the fact that the Brauer indecomposability of
a $p$-permutation bimodule ($p$ is a prime) is one of the key steps in order
to obtain a splendid stable equivalence of Morita type
by making use of the gluing method
due to Brou\'e, Rickard, Linckelmann and Rouquier, that then can possibly
be lifted to a splendid derived (splendid Morita) equivalence.
\end{abstract}


\section{Introduction and notation}
Throughout this paper we denote by $k$ an algebraically closed field
of characteristic $p>0$ and by $G$ a finite group.
For a $p$-permutation $kG$-module $M$ and a $p$-subgroup $P$ of $G$,
the Brauer construction (Brauer quotient) $M(P)$ of $M$ with respect to $P$ plays a very
important role. It canonically becomes a $p$-permutation module
over $kN_G(P)$ (see p.402 in \cite{Bro}).
In their paper \cite{KKM} Kessar, Kunugi and Mitsuhashi introduce a notion {\sl Brauer indecomposability}.
Namely, $M$ is called {\sl Brauer indecomposable} if the restriction
module ${\mathrm{Res}} ^{N_G(Q)}_{Q\,C_G(Q)} [M(Q)]$ is indecomposable or zero
as $k[Q\,C_G(Q)]$-module for any subgroup $Q$ of $P$.
Actually in order to get a kind of equivalence between
the module categories for the principal block algebras $A$ and $B$
of the group algebras $kG$ and $kH$ respectively
(where $H$ is another finite group), e.g.
in order to prove Brou\'e's abelian defect group conjecture,
we usually first of all
have to face a situation such that $A$ and $B$
are stably equivalent of Morita type.
In order to get the stable equivalence,
we often want to check whether the $k(G\times H)$-Scott module
${\mathrm{Sc}}(G\times H, \Delta P)$ with a vertex
$\Delta P:=\{(u,u)\in P\times P \}$ induces a stable
equivalence of Morita type between $A$ and $B$, where
$P$ is a common Sylow $p$-subgroup of $G$ and $H$
by making use of
the gluing method originally due to Brou\'e (see \cite{Bro94}), and also
Rickard, Linckelmann and Rouquier.
If this is the case, then ${\mathrm{Sc}}(G\times H, \Delta P)$
have to be Brauer indecomposable.
Therefore it should be very important to know whether
the $k(G\times H)$-Scott module
${\mathrm{Sc}}(G\times H, \Delta P)$
is Brauer indecomposable or not.
This is the motivation why we have written this paper.

Actually, our main results are the following:

\begin{Theorem}\label{M1}
Suppose that a finite group $G$ has a semidihedral $2$-subgroup $P$.
Assume further that the fusion system ${\mathcal F}_P(G)$ of $G$ over $P$ is saturated and that
$C_G(Q)$ is $2$-nilpotent for every fully ${\mathcal F}_P(G)$-normalized non-trivial subgroup
$Q$ of $P$. Then the Scott module
${\mathrm{Sc}}(G,P)$ is Brauer indecomposable.
\end{Theorem}

\begin{Theorem}\label{product}
Let $G$ and $G'$ be finite groups with a common Sylow $2$-subgroup $P$ which is a
semidihedral group,
and assume that the two fusion systems of $G$ and $G'$ over $P$ are the same,
namely ${\mathcal F}_P(G)={\mathcal F}_P(G')$.
Then the Scott module
${\mathrm{Sc}}(G \times G', \Delta P)$ is Brauer indecomposable.
\end{Theorem}

These theorems in a sense generalize \cite{KKM, KKL, KL, KL2}, and
there are results on Brauer indecomposability of Scott modules also
in \cite{ KT, T}.

\begin{Notation} Besides the notations explained above we need the following notations and terminology.
In this paper $G$ is always a finite group, and $k$ is an algebraically closed field of characteristic $p>0$.

By a $kG$-module we mean a finitely generated left $kG$-module unless stated otherwise.
For a $kG$-module and a $p$-subgroup $P$ of $G$ the Brauer construction (Brauer quotient) $M(P)$ is
defined as in \S27 of \cite{Th} or p.402 in \cite{Bro}. For such $G$ and $P$ we write $\mathcal F_P(G)$ for
the fusion system (fusion category) of $G$ over $P$ as in I.Definition 1.1 of \cite{AKO}.
For two $kG$-modules $M$ and $L$, we write $L\,|\,M$ if $L$ is (isomorphic to) a direct summand of $M$
as a $kG$-module.
For a subgroup $H$ of $G$, a $kG$-module $M$ and a $kH$-module $N$,
we write ${\mathrm{Res}}^G_H(M)$ for the restriction module
$M$ from $G$ to $H$, and ${\mathrm{Ind}}_H^G(N)$ for the induced module of $N$ from $H$ to $G$.
For a subgroup $H\leq G$ we denote the (Alperin-)Scott $kG$-module with respect to $H$
by ${\mathrm{Sc}}(G,H)$.
By definition, ${\mathrm{Sc}}(G,H)$ is the unique indecomposable direct summand
of ${\mathrm{Ind}}_H^G(k_H)$ which contains $k_G$ in its
top (or equivalently in its socle),
where $k_H$ and $k_G$ are the trivial $kH$- and $kG$-modules, respectively.
We refer the reader to
\S2 of \cite{Bro} and  in \S 4.8.4 of \cite{NT} for further details on Scott modules.

We write $O_{p'}(G)$ for the largest normal ${p'}$-subgroup of $G$,
and $Z(G)$ for the center of $G$.
For $x, y\in G$ we set $x^y:=y^{-1}xy$ and $^y\!x:= yxy^{-1}$.
Further for a subgroup $H$ of $G$ and $g\in G$
we set $^g\!H:=\{\, ^g\!h \, |\, \forall h\in H \}$.
For two groups $K$ and $L$ we write $K\rtimes L$ for a semi-direct product of $K$ by $L$
where $K \vartriangleleft (K\rtimes L)$.
For a positive integer $m$, we mean by $C_m$ and $S_m$
the cyclic group of order $m$ and the symmetric group of degree $m$, respectively.
Furthermore  $Q_{2^m}$ for $m\geq 3$ and {\sf{SD}}$_{2^m}$ for $m\geq 4$ are
the generalized quaternion group of order $2^m$ and the semidihedral group of order $2^m$,
respectively.
In fact more precisely speaking our main character {\sl the semidihedral group} {\sf{SD}}$_{2^n}$ is defined by
$$
\text{\sf {SD}}_{2^n}:= \langle x,y \ | \ x^{2^{n-1}}=y^2=1, \ y^{-1}xy =x^{2^{n-2}-1}\rangle
\text{ where }n\geq 4.
$$
We fix the notations $x$, $y$ and $n$ throughout this paper. We also set $z:=x^{2^{n-2}}$,
and hence $\langle z\rangle = Z({\text{\sf SD}}_{2^n})\cong C_2$, see p.191 of \cite{Gor}.

For the other notations and terminologies, see the books
\cite{NT}, \cite{Gor} and \cite{AKO}.
\end{Notation}

The organization of this paper is as follows.
In \S2 we give several theorems which have been proved, and also lemmas
that are useful of our aim, and already in \S2 we prove our first result Theorem \ref{M1}.
In \S3 we prove our second result Theorem \ref{product}.


\section{Preliminaries/Lemma}
In this section we list several previous results and lemmas which are useful
to prove our main results.
The following results of Ishioka and Kunugi will be used for showing Brauer indecomposability of Scott
module under consideration.

\begin{Theorem}[{Theorem 1.3 of \cite{IK}}]\label{IK1}
Assume that $P$ is a $p$-subgroup of $G$ and
${\mathcal F}_P(G)$ is saturated. Then the following assertions are equivalent:
\begin{enumerate}
\item[\rm (1)] ${\mathrm{Sc}}(G,P)$ is Brauer indecomposable.
\item[\rm (2)] ${\mathrm{Res}}_{Q\,C_G(Q)}^{N_G(Q)} {\mathrm{Sc}}(N_G(Q), N_P(Q))$ is indecomposable for each fully ${\mathcal F}_P(G)$-normalized
subgroup $Q$ of $P$.
\end{enumerate}
If these conditions are satisfied, then $({\mathrm{Sc}}(G,P))(Q)
\cong {\mathrm{Sc}}(N_G(Q),N_P(Q))$ for each fully ${\mathcal F}_P(G)$-normalized
subgroup $Q \leq P$.
\end{Theorem}

\begin{Theorem}[{Theorem 1.4 of \cite{IK}}]\label{IK2}
Assume that $P$ is a $p$-subgroup of $G$ and ${\mathcal F}_P(G)$ is
saturated. Let $Q$ be a fully ${\mathcal F}_P(G)$-normalized subgroup of $P$. Assume further that there exists a subgroup
$H_Q$ of $N_G(Q)$ satisfying the following conditions:
\begin{enumerate}
\item[\rm (1)] $N_P(Q)$ is a Sylow $p$-subgroup of  $H_Q$ and
\item[\rm (2)] $|N_G(Q):H_Q|=p^a$ for an integer $a\geq 0$.
\end{enumerate}
Then ${\mathrm{Res}}_{Q\,C_G(Q)}^{N_G(Q)} {\mathrm{Sc}}(N_G(Q), N_P(Q))$ is indecomposable.
\end{Theorem}

The following lemma is essentially a kind of very nice application of Baer-Suzuki Theorem
Theorem 3.8.2 of \cite{Gor}. It surprisingly does play an important role not only for \cite{KL}
but also for our purpose, though it seems quite elementary. Below, we present a shorter 
proof than the proof given in Lemma 4.2 of \cite{KL}.

\begin{Lemma}[{Lemma 4.2 of \cite{KL}}]\label{S3}
Let $Q$ be a normal $2$-subgroup of $G$ such that $G/Q \cong S_3$.
Assume further that there is an involution $t \in G \backslash Q $. Then $G$ has a subgroup $H$ such that
$t \in H \cong S_3$.
\end{Lemma}

\begin{proof}
We first claim that there is an involution $y\in G \backslash Q$ which is different from $t$.

Assume that there is no such an element.
For any $g\in G$ we have $t^g\,{\not\in}\,Q$ 
since $Q\unlhd G$, and hence $t^g=t$. 
So that $t\in Z(G)$.
Moreover, $\langle Q, t \rangle =Q \langle t \rangle$ and
$$Q \unlhd Q \langle t \rangle \unlhd G.$$
We observe that $Q \neq  Q \langle t \rangle$ and
that $ Q \langle t \rangle \neq G$.
Hence $Q \langle t \rangle /Q \unlhd G/Q\cong S_3$ 
implies that $Q \langle t \rangle /Q \cong C_3$. However, 
$Q \langle t \rangle /Q \cong \langle t \rangle / (\langle t \rangle \cap Q)$ 
is a $2$-group, which is a contradiction.

Now, we can write $G=QK$ where $K=\langle t, y\rangle$. 
Since $K$ is generated by the two involutions $t$ and $y$,
$K$ is dihedral.
Hence there is an element $x\in K$ such that 
$K
= \langle x \rangle \rtimes \langle t \rangle$ and that $x^t=x^{-1}$. 
Since $Q$ is a $2$-group and $G/Q \cong S_3$, 
we have $3\,{|}\,|x|$. So there exists an element $u$ which is 
a power of $x$ and is of order $3$. 
Then $u^t=u^{-1}$, so that we can set $H=\langle u, t\rangle \cong S_3$.
\end{proof}

The following corollary is the corresponding version of Corollary 4.3 in \cite{KL} that will be crucial in the proof of
our main results. Although the proof is quite similar, we add it for completeness.

\begin{Corollary}\label{Q8}
Let $P$ be a semidihedral $2$-subgroup of $G$, and let $Q$ be a subgroup of $P$. Assume that
either $Q \cong C_2 \times C_2$ or $Q \cong Q_8$.
Assume, moreover, that $C_G(Q)$ is $2$-nilpotent and $N_G(Q)/Q\,C_G(Q) \cong S_3.$ Then there exists a subgroup $H_Q$ of
$N_G(Q)$ such that $N_P(Q)$ is a Sylow $2$-subgroup of $H_Q$ and $|N_G(Q) : H_Q|$ is a power of $2$
(possibly $1$).
\end{Corollary}

\begin{proof}
Since $C_G(Q)$ is $2$-nilpotent, the group $Q\,C_G(Q)$ is also $2$-nilpotent. Let $K:=O_{2'}(Q\,C_G(Q))$ and let $R \in {\mathrm{Syl}}_2(Q\,C_G(Q))$ containing
$Q\,C_P(Q)$, so that $Q\,C_G(Q)=K\rtimes R$. Note that since $O_{2'}(Q\,C_G(Q)) =O_{2'}(C_G(Q))$ we have $[K,Q]=1$ and so $K\rtimes Q=K \times Q$.
Note that $(K\times Q) \unlhd (K \rtimes N_P(Q))$. Moreover, since $K$ is a characteristic subgroup of $Q\,C_G(Q)$ and $Q\,C_G(Q) \unlhd N_G(Q)$,
we have $K \unlhd N_G(Q)$, so that $K \times Q \unlhd N_G(Q)$.

Let us consider quotients with respect to $L:=K \times Q$ and use the notation $\overline{H}$ for the image of $H \leq G$
under the natural epimorphism $\pi_L:N_G(Q) \twoheadrightarrow N_G(Q)/L$. Then $(\overline{Q\,C_G(Q)} )\cong R/Q$ is
a normal 2-subgroup of $\overline{N_G(Q)} $ and we have an isomorphism
$$\overline{N_G(Q)} \ / \ (\overline{Q\,C_G(Q)} ) \cong N_G(Q) \ / \ (Q\,C_G(Q)) \cong S_3.$$
Note that, $\overline{(K \rtimes N_P(Q))} \cong N_P(Q)/Q \cong C_2$ where
the latter isomorphism comes from
Lemma (viii) of \cite{KLCP}. Note also that,
$$\overline{(K \rtimes N_P(Q))} \cap \overline{(Q\,C_G(Q))} =\overline{(K\rtimes Q\,C_P(Q))}=\overline{1}$$
where the latter equality is true since $C_P(Q)=Z(Q)$ from Lemma (viii) of \cite{KLCP}. So there is an involution
$t \in \overline{N_G(Q)} \backslash \overline{(Q\,C_G(Q))}$ (in fact there is one in
$\overline{(K \rtimes N_P(Q))}$\ ). Hence, by Lemma \ref{S3} there is a subgroup $H$ of $\overline{N_G(Q)}$ such that $t \in H \cong S_3$. Set $H_Q$
as the preimage of $H$ under $\pi_L$. It is easy to see that $H_Q$ satisfies the required properties.
\end{proof}

Now, we are ready to prove the first main result, namely Theorem \ref{M1}.

\begin{proof}[Proof of Theorem \ref{M1}]
Set $ {\mathcal F}:={\mathcal F}_P(G)$. We want to prove that $M(Q)$ is
indecomposable as $Q\,C_G(Q)$-module.
It follows from Lemma 4.3(ii) of \cite{KKM}
that $M(P)$ is indecomposable as $P\,C_G(P)$-module.
Hence we can assume that $Q \lneqq P$.
From Lemma (i) of \cite{KLCP}, $P$ has exactly three maximal subgroups each of which
is either isomorphic to cyclic,
dihedral or generalized quaternion. Hence a fully ${\mathcal F}_P(G)$-normalized
subgroup $Q$ of $P$ is
isomorphic to either cyclic or dihedral or generalized quaternion.
Thus, from Proposition 3.2 (1) of \cite{CG} or Lemma 2.1.(i) of \cite{O},
it is easy to see that ${\mathrm{Aut}}(Q)$ is a $2$-group
unless $Q$ is isomorphic to $C_2 \times C_2$ or $Q_8$.

Suppose that $Q \not\cong C_2 \times C_2$ or $Q_8$. Although the proof for this case is
similar to the second paragraph of the proof of Theorem 1.3 of \cite{KL}, we add the details for the convenience of the reader.
Since in this case ${\mathrm{Aut}}(Q)$ is a $2$-group, $N_G(Q)/C_G(Q)$ is also a $2$-group.
Thus $N_G(Q)$ is $2$-nilpotent, since $C_G(Q)$ is assumed to be $2$-nilpotent. So
we can write $N_G(Q)=K \rtimes S$ where $K:=O_{2'}(N_G(Q))$ and $S$ is a Sylow $2$-subgroup of $N_G(Q)$.
Since $N_P(Q)$ is a $2$-subgroup, without loss of generality we can assume that $N_P(Q) \leq S$.
Let us set $H_Q:= K\rtimes N_P(Q)$, then $N_P(Q)$ is a Sylow $2$-subgroup of $H_Q$ and $|N_G(Q):H_Q|$ is
a power of $2$.

Suppose that $Q \cong C_2 \times C_2$ or $Q_8$. Then ${\mathrm{Out}}(Q) \cong S_3$ and
$$C_2 \cong N_P(Q)/Q\,C_P(Q) \leq N_G(Q)/Q\,C_G(Q)\hookrightarrow {\mathrm{Out}}(Q)$$
where the former isomorphism follows from Lemma (viii) of \cite{KLCP}, so that
we have two cases. If $N_G(Q) /Q\,C_G(Q)\cong C_2$, we can argue as in the previous paragraph and
get the desired subgroup $H_Q$. If $N_G(Q) /Q\,C_G(Q)\cong S_3$, by Corollary \ref{Q8} we again get $H_Q$
with the desired properties.

Therefore for all possible fully ${\mathcal F}_P(G)$-normalized $Q$, we find the subgroup $H_Q$ satisfying the conditions of
Theorem \ref{IK2}, so we can conclude that ${\mathrm{Sc}}(G,P)$ is Brauer indecomposable.
\end{proof}

\begin{Lemma}\label{cent}
Let $P$ be a subgroup of $G$ which is isomorphic to a semi-
\linebreak dihedral $2$-group
such that ${\mathcal F}_P(G)$ is saturated.
If, furthermore, $Q$ is a fully $\mathcal F_P(G)$-normalized subgroup of $P$ such that
$C_G(Q)$ is $2$-nilpotent,
then ${\mathrm{Res}}_{Q\,C_G(Q)}^{N_G(Q)} {\mathrm{Sc}}(N_G(Q), N_P(Q))$ is indecomposable, namely,
\linebreak
${\mathrm{Res}}_{Q\,C_G(Q)}^{N_G(Q)}\Big([{\mathrm{Sc}}(G, P)]( Q)\Big)$
is indecomposable.
\end{Lemma}

\begin{proof} This follows from the proof of Theorem \ref{M1} and the final part of Theorem \ref{IK1}.
\end{proof}

\section{Proof of Theorem \ref{product}}
The following lemma is used in the proof of our second main result.

{\begin{Lemma}\label{Q>8}
Let $P$ be a Sylow $2$-subgroup of $G$ which is isomorphic to a
semidihedral $2$-group.
If $Q\leq P$ with $|Q| \geq 8$, then $C_G(Q)$ is $2$-nilpotent.
\end{Lemma}

\begin{proof}
Set $P:={\sf SD}_{2^n}:= \langle x,y \ | \ x^{2^{n-1}}=y^2=1, \ x^y=x^{2^{n-2}-1}\rangle$
where $n\geq 4$ as before.
Note that
setting $z:=x^{2^{n-2}}$, we have that $Z(P)= \langle z \rangle \cong C_2$.
If $Q$ contains an element $x^i \not\in \{1, z\}$, then $C_G(Q)$ is
$2$-nilpotent, because from line $-$3 on page 246 of \cite{Bra}
we know that for such $x^i$, $C_G(x^i)$ is $2$-nilpotent. So, in order to prove the
claim, it is enough to show that if $|Q| \geq 8$, then
there is an element $x^i \not\in\{ 1, z\}$ which lies in $Q$.

By Lemma (ii) of \cite{KLCP}, the elements $x^iy$ has order $2$ or $4$, when $i$ is
even or odd, respectively. So if
$Q$ is cyclic, then it must contain $x^i$ which has order greater or equal to $8$,
hence the claim is true. So,
we may assume that $Q$ is generated by two elements. We have two cases for this
situation, either $Q= \langle x^i y, x^j y \rangle $
or $Q= \langle x^\kappa y, x^l  \rangle $ where $i, j, \kappa, l$ are integers in the set $
\{1, 2, \ldots , 2^{n-1} \}$ with $i \neq j$.

In the first case, $Q$ has the element $a_{i,j}:=  (x^i y) (x^j y)=x^{i-j} z^j$
inside itself. If both $i$ and $j$ are even, then since $z^j=1$ we have that
$a_{i,j}=  x^{i-j}$. If $x^{i-j} \not\in \{1, z\}$, we have the result. Otherwise,
since $i \neq j$, $x^{i-j}$ can not be trivial,
hence there is only one possibility left:  $x^{i-j}=z$, but this implies
that $ (x^i y) (x^j y)= (x^j y) (x^i y)=z$, hence $Q \cong C_2 \times C_2$ contradicting
with the assumption $|Q| \geq 8$. So  $x^{i-j}=z$ can not happen.
If one of $i$ or $j$
is even and the other one is odd, then $a_{i,j}=x^{i-j} z^j$ where $i-j$ is odd. So, in this
case $a_{i,j}$ is a non-trivial element in $\langle x \rangle \backslash Z(P)$,
whence the result. If both $i$ and $j$ are odd, then $a_{i,j}=  x^{i-j} z$.
Similarly, if $x^{i-j} \not\in \{1, z\}$, we have the result. Otherwise, since $i \neq j$,
there is only
one possibility left:  $x^{i-j}=z$. In this case, $(x^i y)(x^j y)=z.z=1$ so that
$x^i y=(x^j y)^{-1}$. But this is impossible because if this holds $Q \cong C_4$
contradicting our assumption. So  $x^{i-j}=z$ can not happen in this case, too.

In the second case, if $x^l$ lies outside $Z(P)$, the result is true. So let us assume
that $x^l=z$ (since $x^l$ is non-tirivial by the assumptions),
then $Q=\langle x^\kappa y, z \rangle$. If $\kappa$ is even then $x^\kappa y$
is an involution and commutes with the central involution $z$, hence
$\langle x^\kappa y, z \rangle \cong C_2 \times C_2$. This gives a contradiction with our
assumption on $Q$. If $\kappa$ is odd, then $(x^\kappa y)^2=z$, so $Q \cong C_4$, which is
similarly a contradiction with the property of $Q$.
\end{proof}

\begin{proof}[Proof of Theorem \ref{product}]
Set $\mathcal H:=G \times G'$. Since $P$ is a Sylow $2$-subgroup of $G$,  ${\mathcal F}:={\mathcal F}_P(G)$ is a saturated fusion system
(see Proposition 1.3 of \cite{BLO}). Moreover since
${\mathcal F}_{\Delta P}(\mathcal H) \cong {\mathcal F}_P(G)$, we have that ${\mathcal F}_{\Delta P}(\mathcal H) $ is saturated.

Let $\Delta Q \leq \Delta P$ be any fully ${\mathcal F}_{\Delta P}(\mathcal H) $-normalized subgroup of $\Delta P$.
Setting $M:={\mathrm{Sc}}(\mathcal H, \Delta P)$, we claim that $M(\Delta Q)$ is
indecomposable as a $\Delta Q \, C_{\mathcal H}(\Delta Q)$-module.
Note that $C_{\mathcal H}(\Delta Q)=C_G(Q) \times C_{G'}(Q)$.

{\bf Case 1: $|Q|\geq 8$:}
If $Q$ is a subgroup of $P$ such that $|Q| \geq 8$, by Lemma \ref{Q>8} both $C_G(Q)$ and
$C_{G'}(Q)$ are $2$-nilpotent. Hence, $C_{\mathcal H}(\Delta Q)$ is also $2$-nilpotent.
Therefore, by Lemma \ref{cent} we have that
${\mathrm{Res}}_{\Delta Q \, C_{\mathcal H}(\Delta Q)}^{N_{\mathcal H}(\Delta Q)} M(\Delta Q)$ or equivalently
$${\mathrm{Res}}_{\Delta Q \, C_{\mathcal H}(\Delta Q)}^{N_{\mathcal H}(\Delta Q)}
{\mathrm{Sc}}(N_{\mathcal H}(\Delta Q), N_{\Delta P}(\Delta Q))$$
is indecomposable for such $Q$.

{\bf Case 2:} Either $Q\cong C_2\times C_2$ or $Q \cong C_4$:
Let $Q$ be a fully ${\mathcal F}$-normalized subgroup of $P$ such that $Q\cong C_2\times C_2$.
Since by Lemma (viii) of \cite{KLCP} there is only one $P$-conjugacy class of Klein four subgroups of $P$,
we can assume $Q:=Z(P)\times\langle y\rangle$.
Then, again by Lemma (viii) of \cite{KLCP}, $C_P(Q)=Q$, and
by Proposition 2.5 of \cite{L}, $Q$ is fully
${\mathcal F}$-centralized, so that from Lemma 2.10(i) of \cite{L}, $Q$ is a Sylow $2$-subgroup of $C_G(Q)$.
Hence the Schur-Zassenhaus Theorem yields that $C_G(Q)=Q\times O_{2'}(C_G(Q))$,
so that $C_G(Q)$ is $2$-nilpotent. Similarly, $C_{G'}(Q)$ is also $2$-nilpotent.
Thus, $C_{\mathcal H}(\Delta Q)$ is also $2$-nilpotent. Hence Lemma \ref{cent} implies that
${\mathrm{Res}}_{C_{\mathcal H}(\Delta Q)}^{N_{\mathcal H}(\Delta Q)}[M(\Delta Q)]$ is indecomposable.

Now, let $Q$ be a fully ${\mathcal F}$-normalized subgroup of $P$ such that $Q\cong C_4$. By
Theorem of \cite{KLCP} either all elements of order $4$ are $G$-conjugate or there are
exactly two $G$-conjugacy classes of elements of order $4$. The same situation holds for $G'$-conjugacy classes
and hence for ${\mathcal F}$-conjugacy classes. In the first case,
$\langle x^{2^{n-3}} \rangle$ is a fully ${\mathcal F}$-normalized representative of the single ${\mathcal F}$-conjugacy
class of cyclic subgroups of order $4$ of $P$. In the second
case, $\langle x^{2^{n-3}} \rangle$ and $\langle xy \rangle$ are representatives of fully ${\mathcal F}$-normalized
subgroups of two ${\mathcal F}$-conjugacy classses of cyclic subgroups of order $4$ of $P$. If $Q=\langle x^{2^{n-3}} \rangle$, then  line $-$3 on page 246 of \cite{Bra}
implies that both $C_G(Q)$ and $C_{G'}(Q)$ are $2$-nilpotent and hence $C_{\mathcal H}(\Delta Q )$
is a $2$-nilpotent subgroup of $\mathcal H$. Therefore Lemma \ref{cent} implies that
${\mathrm{Res}}_{C_{\mathcal H}(\Delta Q)}^{N_{\mathcal H}(\Delta Q)}[M(\Delta Q)]$ is indecomposable.
If $Q=\langle xy \rangle$, then since $Q$ is fully ${\mathcal F}$-normalized, it is fully ${\mathcal F}$-centralized
(see Proposition 2.5 of \cite{L}),
and hence by Lemma 2.10 (i) of \cite{L}, $C_P(Q) \in {\mathrm{Syl}}_2(C_G(Q))$, and similarly $C_P(Q) \in {\mathrm{Syl}}_2(C_{G'}(Q))$.
From Lemma (ix) of \cite{KLCP}, $C_P(Q)=Q$, and hence $C_P(Q)$ is a normal abelian Sylow $2$-subgroup of both $C_G(Q)$ and
$C_{G'}(Q)$. Hence the Schur-Zassenhaus Theorem implies that $C_G(Q)=Q \times O_{2'}(C_G(Q))$ and similarly for $C_{G'}(Q)$. In particular,
both of them are $2$-nilpotent.
Therefore $C_{\mathcal H}(\Delta Q)$ is $2$-nilpotent. Thus by Lemma \ref{cent}, we have that
${\mathrm{Res}}_{C_{\mathcal H}(\Delta Q)}^{N_{\mathcal H}(\Delta Q)}[M(\Delta Q)]$ is indecomposable in this
case, too.

\smallskip

{\bf Case 3: $Q \cong C_2$:}
Let $Q$ be a subgroup of $P$ which is isomorphic to a cyclic group of order $2$. From
Theorem in \cite{KLCP}, either all
involutions in $P$ are $G$-conjugate or there are two $G$-conjugacy classes of involutions of $P$ represented by $z$
and $y$. Similar situation holds for $G'$-conjugacy classes and for ${\mathcal F}$-conjugacy classes. Hence, it is enough to
prove the desired result for $Q=Z(P)$ and for $Q=\langle y \rangle$, since they are the fully ${\mathcal F}$-normalized elements in their own ${\mathcal F}$-conjugacy classes.

Suppose first that $Q=Z(P)$,
then $\Delta P \leq N_{\mathcal H}(\Delta Q)$. Suppose also that $M(\Delta Q)=M_1 \oplus \ldots \oplus M_r$
where $r \geq1 $ and $M_i$ are indecomposable  $N_{\mathcal H}(\Delta Q)$-modules.
From Theorem 4.8.6 (ii)] of \cite{NT} we can set $M_1:= {\mathrm{Sc}}(N_{\mathcal H}(\Delta Q), \Delta P)$. Since
$M(\Delta Q) \ | \ {\mathrm{Res}}_{N_{\mathcal H}(\Delta Q)}^{\mathcal H} (M) $,
we have that
$ M_i\ | \ {\mathrm{Res}}_{N_{\mathcal H}(\Delta Q)}^{\mathcal H} (M) $ for each $i$. Moreover
since $M \ | \ \mathrm{Ind}_{\Delta P}^{\mathcal H}(k)$, by Mackey decomposition we have that for a fixed $i\geq 2$,
$$M_i \ | \ \bigoplus_h {\mathrm{Ind}}_{N_{\mathcal H}(\Delta Q) \cap \ ^h\! \Delta P}^{N_{\mathcal H}(\Delta Q)} (k)$$
where
$h$ runs over representatives
of the double cosets in $N_{\mathcal H}(\Delta Q) \backslash {\mathcal H} / \Delta P$
which satisfies $\Delta Q \leq \ ^h\! \Delta P$ by 1.4 of \cite{Bro}. Hence a vertex $\Delta R$
of $M_i$ lies in $N_{^h\!\Delta P}(\Delta Q)=N_{\mathcal H}(\Delta Q) \cap \ ^h\! \Delta P$ for some $h$.
Note that
for such $h$, we have that
$$\Delta R \leq N_{^h\!\Delta P}(\Delta Q) \leq_{N_{\mathcal H}(\Delta Q)}  \Delta P$$
by Lemma 3.2 of \cite{IK}. So, we have that $M_1(\Delta R) \neq 0$.
On the other hand,
by applying Burry-Carlson-Puig's Theorem for $\Delta Q$
(see Theorem 4.4.6 (ii)] of \cite{NT}),
each $M_i$ has vertex not equal to $\Delta Q$. Hence $\Delta Q $ is a proper normal subgroup of $\Delta R$ and so
$$ M_1(\Delta R) \oplus M_i(\Delta R) \ | \  (M(\Delta Q))(\Delta R) \cong \ M(\Delta R)$$
as $N_{\mathcal H}(\Delta R) \cap N_{\mathcal H}(\Delta Q)$-modules, but
$C_{\mathcal H}(\Delta R) \leq N_{\mathcal H}(\Delta R) \cap N_{\mathcal H}(\Delta Q)$, so
the isomorphism above restricts to as  $C_{\mathcal H}(\Delta R) $-modules which contradicts with the fact that
$M(\Delta R)$ is indecomposable as a  $C_{\mathcal H}(\Delta R) $-module by Case 1 and Case 2.
Therefore, $r=1$ and $M(\Delta Q)$ is indecomposable
as a $C_{\mathcal H}(\Delta Q)$-module
since $C_{\mathcal H}(\Delta Q)=N_{\mathcal H}(\Delta Q)$.

Suppose now that $Q=\langle y \rangle$, and set
$S:=C_P(Q)$. Then,
$S=C_P(Q)=N_P(Q)=\langle y , z\rangle \cong C_2 \times C_2$.
Assume similarly that $M(\Delta Q)=M_1 \oplus \ldots \oplus M_r$
 where $M_i$ are indecomposable  $N_{\mathcal H}(\Delta Q)$-modules. By
Lemma 3.1 of \cite{IK} and Theorem 1.7 of \cite{K} we can set
$M_1:={\mathrm{Sc}}(N_{\mathcal H}(\Delta Q), \Delta S)$. Since $M \,|\,  {\mathrm{Ind}}_{ \Delta P}^{\mathcal H} (k) $, we have that
$$ M(\Delta Q) \ \Big| \  \bigoplus_{h} {\mathrm{Ind}}_{N_{\mathcal H}(\Delta Q) \cap \ ^h\! \Delta P}^{N_{\mathcal H}(\Delta Q)} (k)$$ where
$h$ runs over representatives
of the double cosets in $N_{\mathcal H}(\Delta Q) \backslash {\mathcal H} / \Delta P$
which satisfies $\Delta Q \leq \ ^h\! \Delta P$ by 1.4 of \cite{Bro}.
Hence,
$$M_i  \ \Big| \  \bigoplus_{h} {\mathrm{Ind}}_{N_{\mathcal H}(\Delta Q)
\cap \ ^h\! \Delta P}^{N_{\mathcal H}(\Delta Q)} (k)$$
for each $i$, where $h$ runs through the same
set. Hence, for a fixed $i$, a vertex of $M_i$ is contained in $N_{^h\!\Delta P}(\Delta Q)$. Note that
for such $h$, we have that
$$N_{^h\!\Delta P}(\Delta Q) \leq_{N_{\mathcal H}(\Delta Q)}  \Delta S$$
by Lemma 3.2 of \cite{IK}. So for a fixed $i$, any vertex of $M_i$ is
contained in $\Delta S$.
On the other hand, by Lemma 2.1 of \cite{IK} and
Burry-Carlson-Puig's Theorem (see Theorem 4.4.6 (ii) in \cite{NT}),
a vertex of $M_i$
contains $\Delta Q$ properly. Hence each $M_i$ has a vertex
$\Delta S$ { (Note that $|S:Q|=2$)}.
This implies that $M_i(\Delta S) \not= 0$ by (27.7) Corollary  in \cite{Th}.
Therefore, since $\Delta Q$ is normal in $\Delta S$,
$$ M(\Delta S) \cong (M(\Delta Q))(\Delta S)
= M_1(\Delta S) \oplus \ldots \oplus M_r(\Delta S)$$
as $N_{\mathcal H}(\Delta S) \cap N_{\mathcal H}(\Delta Q)$-modules. Since
$C_{\mathcal H}(\Delta S) \leq N_{\mathcal H}(\Delta S)
\cap N_{\mathcal H}(\Delta Q)$
and since we have already proved in {\bf Case 2} that
$M(\Delta S)$ is indecomposable as a $C_{\mathcal H}(\Delta S)$-module
(recall that $S\cong C_2\times C_2$), we have $r=1$, so $M(\Delta Q)$ is an indecomposable
$C_{\mathcal H}(\Delta Q)$-module.

\end{proof}

\bigskip

{\bf{Acknowledgements.}}
{\small
The first author was supported by
the Japan \linebreak
Society for Promotion of Science (JSPS), Grant-in-Aid for Scientific
\linebreak Research
(C)19K03416, 2019--2021. The second author was supported by Mimar Sinan Fine Arts University
Scientific Research Unit with project number 2019-28}.
The authors thank Caroline Lassueur for useful \linebreak
information on \cite{CG}.
}

\end{document}